\documentclass[jkps,preprint,fleqn,showpacs,showkeys]{revtex4}
\usepackage{graphicx}
\usepackage{amssymb}
\usepackage{amsmath,amsthm}
\usepackage{bm}
\begin{document}
\setcounter{page}{0}

\newcommand{\eps}{\varepsilon}
\newcommand{\set}[1]{\left\{#1\right\}}
\newcommand{\abs}[1]{\left|#1\right|}
\newcommand{\p}{\partial}
\newcommand{\mB}{\mathbf{B}}
\newcommand{\mE}{\mathbf{E}}
\newcommand{\mP}{\mathbf{P}}
\newcommand{\mU}{\mathbf{U}}
\newcommand{\mV}{\mathbf{V}}
\newcommand{\mW}{\mathbf{W}}
\newcommand{\mc}{\mathbf{c}}
\newcommand{\md}{\mathbf{d}}
\newcommand{\me}{\mathbf{e}}
\newcommand{\mf}{\mathbf{f}}
\newcommand{\mg}{\mathbf{g}}
\newcommand{\mn}{\mathbf{n}}
\newcommand{\mt}{\mathbf{t}}
\newcommand{\mx}{\mathbf{x}}
\newcommand{\my}{\mathbf{y}}
\newcommand{\mz}{\mathbf{z}}
\newcommand{\vg}{\boldsymbol{\gamma}}
\newcommand{\vn}{\boldsymbol{\nu}}
\newcommand{\vp}{\boldsymbol{\phi}}
\newcommand{\vt}{\boldsymbol{\theta}}
\newcommand{\vv}{\boldsymbol{\vartheta}}
\newcommand{\vx}{\boldsymbol{\xi}}

\newtheorem{thm}{Theorem}[section]
\newtheorem{cor}[thm]{Corollary}
\newtheorem{lem}[thm]{Lemma}
\newtheorem{prop}[thm]{Proposition}
\newtheorem{defn}[thm]{Definition}
\newtheorem{rem}[thm]{Remark}

\title[]{Detection of small electromagnetic inhomogeneities with inaccurate frequency}
\author{Won-Kwang \surname{Park}}
\email{parkwk@kookmin.ac.kr}
\thanks{Fax: +82-2-910-4739}
\affiliation{Department of Mathematics, Kookmin University, Seoul
02707}

\date[]{Received 6 August 2007}

\begin{abstract}
Generally, in the application of subspace migration for detecting locations of small inhomogeneities, one begins reconstruction procedure with \textit{a priori} information of applied frequency. However, mathematical theory of subspace migration has not been developed satisfactorily  when applied frequency is unknown. In this paper, we identify mathematical structure of subspace migration imaging function for finding locations of small inhomogeneities in two-dimensional homogeneous space by establishing a relationship with Bessel functions of integer order zero and one of the first kind. This expression indicates the reason behind the appearance of inaccurate locations. Numerical simulations are performed to support our analysis.
\end{abstract}



\maketitle

\section{INTRODUCTION}
Generally, one of the purpose of inverse scattering problem is to identify locations of small electromagnetic inhomogeneities from measured scattered field or far-field pattern. This problem is known as a difficult problem due to the its nonlinearity and ill-posedness but still interesting and challengeable problem because this arises in Mathematics, Physics, Medical imaging, Engineering sciences, etc, highly related to the modern life. Related works can be found in \cite{A3,ABF,FKM,KDAK,SKLKLJC} and references therein.

Motivated this, various algorithms for solving inverse scattering problem have been developed. Following various researches \cite{AK2,BL,BK,B3,CGHIR,DL,GK,IW}, most of which are based on the least-square method so, for guarantee a successful performance, \textit{a priori} information of unknown inhomogeneities, appropriate regularization terms highly depends on the specific problems, calculation of complex Fr\'echet (or domain) derivative must be considered beforehand. If any one of these conditions is not fulfilled, serious problems such as non-convergence, the local minimizer problem, and a considerable increase in the computational costs due to the large number of iteration procedures will arise.

For an alternative, fast identification algorithms have been developed. Among them, single- and multi-frequency Kirchhoff and subspace migration have shown their feasibilities in detection of small inhomogeneities for full- and limited-view inverse scattering problems, refer to \cite{AGKPS,P-SUB3,PL2}. However, exact value of the applied frequency must be known in order to detect locations of inhomogeneities accurately. If not, it would only be possible to recognize the existence of inhomogeneities, i.e,. identification of exact locations of inhomogeneities is impossible. Throughout various simulation results, this fact has been examined (see \cite{LVB}) and recently, related mathematical theory of Multiple Signal Classification (MUSIC) algorithm for detecting small electromagnetic inhomogeneities has been concerned (see \cite{SRACM}); however, reliable mathematical theory has not yet been developed satisfactorily.

In this paper, we carefully analyze subspace migration imaging function with inaccurate frequency by establishing a relationship with Bessel functions of order zero and one of the first kind. This is based on the asymptotic expansion formula in the presence of a set of electromagnetic inhomogeneities with small diameter and the structure of singular vectors associate with the nonzero singular values of the so-called Multi-Static Response (MSR) matrix collected from the far-field pattern. The identified relationship explains why the subspace migration yields inaccurate locations of small inhomogeneities with inaccurate frequency.

Remaining parts of this paper is organized as follows. In Section \ref{sec:2}, we introduce the direct scattering problem, far-field pattern, and asymptotic expansion formula. In Section \ref{sec:3}, the subspace migration algorithm for detection of small inhomogeneities is surveyed. In Section \ref{sec:4}, we establish a relationship between subspace migration imaging function and Bessel functions, and investigate the cause of inaccurate results. In Section \ref{sec:5}, the results of the numerical simulations are exhibited in support of our analysis. A short conclusion follows in Section \ref{sec:6}.

\section{Two-dimensional direct scattering problem}\label{sec:2}
Let $\Sigma_m$ be a homogeneous inclusion with a small diameter $r$ in the two-dimensional space $\mathbb{R}^2$. Throughout this paper, we assume totally $M$ different small inhomogeneities $\Sigma_m$ exist in $\mathbb{R}^2$ such that
\[\Sigma_m=\mz_m+r\mathbf{B}_m,\]
where $\mz_m$ denote the location of $\Sigma_m$ and $\mathbf{B}_m$ is a simply connected smooth domain containing the origin. For the sake, we assume that $\mathbf{B}$ is a unit circle and $\Sigma_m$ are separated from each other.

Let $\omega=2\pi/\lambda$ be a given positive angular frequency with $\lambda$ denotes given wavelength. Throughout this paper $\omega$ is sufficiently large enough and satisfying following condition
\begin{equation}\label{Separated}
\omega|\mz_m-\mz_{m'}|\gg0.75\quad\mbox{and}\quad\lambda>2r
\end{equation}
for all $m,m'=1,2,\cdots,M$ with $m\ne m'$.

Throughout this paper, we denote $\eps_0$ and $\mu_0$ be the dielectric permittivity and magnetic permeability of $\mathbb{R}^2$, respectively. Similarly, we let $\eps_m$ and $\mu_m$ be those of $\Sigma_m$. For simplicity, let $\Sigma$ be the collection of $\Sigma_m$, $m=1,2,\cdots,M,$ and we define the following piecewise constants:
\[\eps(\mx)=\left\{\begin{array}{ccl}
\medskip\eps_m&\mbox{for}&\mx\in\Sigma_m\\
\eps_0&\mbox{for}&\mx\in\mathbb{R}^2\backslash\overline{\Sigma}
\end{array}\right.
\quad\mbox{and}\quad
\mu(\mx)=\left\{\begin{array}{ccl}
\medskip\mu_m&\mbox{for}&\mx\in\Sigma_m\\
\mu_0&\mbox{for}&\mx\in\mathbb{R}^2\backslash\overline{\Sigma}.
\end{array}\right.\]
In this paper, we assume that $\eps_0=\mu_0=1$, $\eps_m\ne\eps_0$, and $\mu_m\ne\mu_0$ for the sake of simplicity.

At a given positive angular frequency $\omega$, let $u(\mx)$ be the time-harmonic total field that satisfies the Helmholtz equation
\begin{equation}\label{HelmHoltzEquation}
  \nabla\cdot\bigg(\frac{1}{\mu(\mx)}\nabla u(\mx)\bigg)+\omega^2\eps(\mx)u(\mx)=0
\end{equation}
with transmission conditions on $\Sigma_m$ for all $m$.

Let $u_0(\mx)$ be the solution of (\ref{HelmHoltzEquation}) without $\Sigma$. In this paper, we consider the following plane-wave illumination: for a vector $\vt\in\mathbb{S}^1$, $u_0(\mx)=e^{i\omega\vt\cdot\mx}$. Here, $\mathbb{S}^1$ denotes a two-dimensional unit circle and throughout this paper, we assume that the set $\set{\vt_n:n=1,2,\cdots,N}$ spans $\mathbb{S}^1$.

Generally, the total field $u(\mx)$ can be divided into the incident field $u_0(\mx)$ and the unknown scattered field $u_{\mathrm{S}}(\mx)$, which satisfies the Sommerfeld radiation condition
\[\lim_{|\mx|\to\infty}\sqrt{|\mx|}\left(\frac{\p u_{\mathrm{S}}(\mx)}{\p|\mx|}-ik_0u_{\mathrm{S}}(\mx)\right)=0,\quad k_0=\omega\sqrt{\eps_0\mu_0}\]
uniformly in all directions $\hat{\mx}=\frac{\mx}{|\mx|}$. Notice that we assumed $\eps_0=\mu_0=1$ so, we set $k_0=\omega$ from now on. As given in \cite{AK2}, $u_{\mathrm{S}}$ can be written as the following asymptotic expansion formula in terms of $r$
\begin{multline}\label{ScatteredField}
  u_{\mathrm{S}}(\mx)=r^2\sum_{m=1}^{M}\bigg(\nabla u_0(\mx)(\mz_m)\cdot\mathbb{A}(\mz_m)\cdot\nabla\Phi(\mz_m,\mx)\\+\omega^2(\eps-\eps_0)u_0(\mx)(\mz_m)\Phi(\mz_m,\mx)\bigg)+o(r^2),
\end{multline}
where $o(r^2)$ is uniform in $\mz_m\in\Sigma_m$ and $\vt\in\mathbb{S}^1$, $\mathbb{A}(\mz_m)$ is a $2\times2$ symmetric matrix defined as
\begin{align*}
\mathbb{A}(\mz_m)=\frac{2\mu_0}{\mu_m+\mu_0}\left[\begin{array}{cc}1&0\\0&1\end{array}\right]\mbox{area}(\mathbf{B}_m),
\end{align*}
and $\Phi(\mz_m,\mx)$ is the two-dimensional time harmonic Green function (or fundamental solution to Helmholtz equation)
\[\Phi(\mz_m,\mx)=-\mu_0\frac{i}{4}H_0^1(\omega|\mz_m-\mx|).\]
Here, $H_0^1$ is the Hankel function of order zero and of the first kind.

The far-field pattern is defined as function $u_\infty(\hat{\mx},\vt)$ that satisfies
\begin{equation}\label{FarField}
  u_{\mathrm{S}}(\mx)=\frac{e^{i\omega|\mx|}}{\sqrt{|\mx|}}u_\infty(\hat{\mx},\vt)+o\left(\frac{1}{\sqrt{|\mx|}}\right)
\end{equation}
as $|\mx|\longrightarrow\infty$ uniformly on $\hat{\mx}=\frac{\mx}{|\mx|}$.

\section{Introduction to single-frequency subspace migration}\label{sec:3}
Subspace migration algorithm for identifying locations of small defects introduced in \cite{AGKPS} used the structure of a singular vector of the Multi-Static Response (MSR) matrix $\mathbb{M}=[M_{jl}]_{j,l=1}^{N}=[u_\infty(\vv_j,\vt_l)]_{j,l=1}^{N}$. Note that by combining (\ref{ScatteredField}), (\ref{FarField}), and the asymptotic behavior of the Hankel function,
\[H_0^1(\omega|\mz_m-\mx|)=\frac{e^{i\frac{\pi}{4}}}{\sqrt{8\omega\pi}}\frac{e^{i\omega|\mx|}}{\sqrt{|\mx|}}e^{-i\omega\frac{\mx}{|\mx|}\cdot\mz_m}+o\left(\frac{1}{\sqrt{|\mx|}}\right),\]
the far-field pattern $u_\infty(\vv_j,\vt_l)$ with observation number $j$ and incident wave number $l$ can be represented as
\begin{equation}\label{AFS}
  u_\infty(\vv_j,\vt_l)\approx r^2\omega^2\frac{e^{i\frac{\pi}{4}}}{\sqrt{8\omega\pi}}\sum_{m=1}^{M}\left(\frac{\eps_m-\eps_0}{\sqrt{\eps_0\mu_0}}\mbox{area}(\mathbf{B}_m)-\vv_j\cdot\mathbb{A}(\mz_m)\cdot\vt_l\right)e^{i\omega(\vt_l-\vv_j)\cdot\mz_m}.
\end{equation}

Now, let us assume that the incident and observation direction configurations are same, i.e., for each $\vv_j=-\vt_j$. Then,
\[M_{jl}\approx r^2\omega^2\frac{e^{i\frac{\pi}{4}}}{\sqrt{8\omega\pi}}\sum_{m=1}^M\bigg[\frac{\eps_m-\eps_0}{\sqrt{\eps_0\mu_0}}+\frac{2\mu_0}{\mu_m+\mu_0}\vt_j\cdot\vt_l\bigg]\mbox{area}(\mathbf{B}_m)e^{i\omega(\vt_j+\vt_l)\cdot\mz_m}.\]
Then, based on the above observation, we can decompose $\mathbb{M}$ as follows:
\begin{equation}\label{Decomposition}
\mathbb{M}=\mathbb{EFE}^T=r^2\omega^2N\frac{e^{i\frac{\pi}{4}}}{\sqrt{8\omega\pi}}\sum_{m=1}^{M}\mathbb{E}_m(\omega)\left[\begin{array}{ccc}
\displaystyle\frac{\eps_m-\eps_0}{\sqrt{\eps_0\mu_0}}&0&0\\
0&\displaystyle\frac{2\mu_0}{\mu_m+\mu_0}&0\\
0&0&\displaystyle\frac{2\mu_0}{\mu_m+\mu_0}\end{array}\right]\mathbb{E}_m(\omega)^T,
\end{equation}
where
\begin{align}
\begin{aligned}\label{ED}
  \mathbb{E}_m(\omega)&=\frac{1}{\sqrt{N}}\bigg[\mathbf{E}_1(\mz_m;\omega),\mathbf{E}_2(\mz_m;\omega),\mathbf{E}_3(\mz_m;\omega)\bigg]\\
    &=\frac{1}{\sqrt{N}}\left[\begin{array}{ccc}
  e^{i\omega\vt_1\cdot\mz_m}&(\vt_1\cdot\me_1)e^{i\omega\vt_1\cdot\mz_m}&(\vt_1\cdot\me_2)e^{i\omega\vt_1\cdot\mz_m}\\
  e^{i\omega\vt_2\cdot\mz_m}&(\vt_2\cdot\me_1)e^{i\omega\vt_2\cdot\mz_m}&(\vt_2\cdot\me_2)e^{i\omega\vt_1\cdot\mz_m}\\
  \vdots&\vdots&\vdots\\
  e^{i\omega\vt_N\cdot\mz_m}&(\vt_N\cdot\me_1)e^{i\omega\vt_N\cdot\mz_m}&(\vt_N\cdot\me_2)e^{i\omega\vt_1\cdot\mz_m}\end{array}\right]
\end{aligned}
\end{align}
and $\me_1=[1,0]^T$, $\me_2=[0,1]^T$. This decomposition leads us to introduce subspace migration for detecting locations of $\Sigma_m$ as follows. First, let us perform the Singular Value Decomposition (SVD) as follows:
\begin{equation}\label{SVD}
  \mathbb{M}=\mathbb{U}\Lambda\mathbb{\overline{V}}^T=\sum_{m=1}^{N}\sigma_m\mathbf{U}_m\overline{\mathbf{V}}_m^T\approx\sum_{m=1}^{3M}\sigma_m\mathbf{U}_m\overline{\mathbf{V}}_m^T,
\end{equation}
where $\mathbf{U}_m$ and $\mathbf{V}_m$ are the left and right singular vectors of $\mathbb{M}$, respectively. Then, by comparing (\ref{Decomposition}) and (\ref{SVD}), we introduce the following test vector
\begin{equation}\label{VecW}
  \mW(\mx;\omega):=\bigg[\mc_1\cdot[1,\vt_1]^Te^{i\omega\vt_1\cdot\mx},\mc_3\cdot[1,\vt_2]^Te^{i\omega\vt_2\cdot\mx}\cdots,\mc_n\cdot[1,\vt_N]^Te^{i\omega\vt_N\cdot\mx}\bigg]^T.
\end{equation}
Notice that since $\mU_m$ and $\mV_m$ are unit vectors, we introduce following unit vector
\[\hat{\mW}(\mx;\omega)=\frac{\mW(\mx;\omega)}{|\mW(\mx;\omega)|},\]
where $\mc_n\in\mathbb{R}^3\backslash\{\mathbf{0}\}$. Then, based on the orthonormal property of singular vectors, we can easily observe that for a proper choice of $\mc_n$,
\begin{align*}
&\langle\hat{\mW}(\mx;\omega),\mU_m\rangle\approx1\quad\mbox{and}\quad\langle\hat{\mW}(\mx;\omega),\overline{\mV}_m\rangle\approx1\quad\mbox{if}\quad\mx=\mz_m\\
&\langle\hat{\mW}(\mx;\omega),\mU_m\rangle\approx0\quad\mbox{and}\quad\langle\hat{\mW}(\mx;\omega),\overline{\mV}_m\rangle\approx0\quad\mbox{if}\quad\mx\ne\mz_m,
\end{align*}
where $\langle\mathbf{a},\mathbf{b}\rangle:=\overline{\mathbf{a}}\cdot\mathbf{b}$. Correspondingly, we can introduce following imaging function:
\begin{equation}\label{ImagingFunction}
  \mathbb{W}(\mz;\omega):=\left|\sum_{m=1}^{3M}\langle\hat{\mW}(\mx;\omega),\mU_m\rangle\langle\hat{\mW}(\mx;\omega),\overline{\mV}_m\rangle\right|.
\end{equation}
The value of $\mathbb{W}(\mz;\omega)$ will be close to $1$ at $\mx=\mz_m\in\Sigma_m$ and close to $0$ at $\mx\in\mathbb{R}^2\backslash\overline{\Sigma}$. This is the subspace migration algorithm for identifying locations of small inhomogeneities.

\section{Subspace migration with application of inaccurate frequency}\label{sec:4}
Based on the observation discussed in Section \ref{sec:3}, accurate value of $\omega$ must be applied to (\ref{VecW}) for identifying locations of $\Sigma_m$ accurately. If not, it has been heuristically confirmed that inaccurate locations of $\Sigma_m$ are detected via subspace migration. From now on, we analyze structure of (\ref{ImagingFunction}) and explain why this phenomenon occurs. Before starting, we assume that $\omega$ is unknown so, a fixed positive value $\eta$ applied to (\ref{VecW}) instead of $\omega$. Since, we have no \textit{a priori} information of targets, we cannot select optimal vectors $\mc_n$ in (\ref{VecW}). So, motivated from recent works \cite{P-SUB3}, we apply an unit vector $\mW(\mx;\eta)$ instead of (\ref{VecW}) such that
\[\mW(\mx;\eta)=\frac{1}{\sqrt{N}}\bigg[e^{i\eta\vt_1\cdot\mx},e^{i\eta\vt_2\cdot\mx}\cdots,e^{i\eta\vt_N\cdot\mx}\bigg]^T\]
and consider the following imaging function:
\begin{equation}\label{ImagingFunction2}
  \mathbb{W}(\mx;\eta):=\left|\sum_{m=1}^{3M}\langle\mW(\mx;\eta),\mU_m\rangle\langle\mW(\mx;\eta),\overline{\mV}_m\rangle\right|.
\end{equation}

For identifying mathematical structure of (\ref{ImagingFunction2}), we introduce following useful result.
\begin{lem}\label{Bessel}
  Let $\vx\in\mathbb{R}^2$ and $\vt_n\in\mathbb{S}^1$, $n=1,2,\cdots,N$. Then for sufficiently large $N$, the following relations hold:
  \begin{align*}
    &\frac{1}{N}\sum_{n=1}^{N}e^{i\omega\vt_n\cdot\mx}=\frac{1}{2\pi}\int_{\mathbb{S}^1}e^{i\omega\vt\cdot\mx}d\vt=J_0(\omega|\mx|),\\
    &\frac{1}{N}\sum_{n=1}^{N}(\vt_n\cdot\vx)e^{i\omega\vt_n\cdot\mx}=\frac{1}{2\pi}\int_{\mathbb{S}^1}(\vt\cdot\vx)e^{i\omega\vt\cdot\mx}d\vt=i\left(\frac{\mx}{|\mx|}\cdot\vx\right)J_1(\omega|\mx|),
  \end{align*}
  where $J_n(\cdot)$ denotes the Bessel function of integer order $n$ of the first kind.
\end{lem}

Now, we state the main result of this paper.
\begin{thm}\label{TheoremFrequency}
  Assume that total number of incident and observation directions $N$ is sufficiently large enough. Then, $\mathbb{W}(\mx;\eta)$ can be represented as follows:
  \begin{equation}\label{StructureImagingfunction}
  \mathbb{W}(\mx;\eta)=\left|\sum_{m=1}^{M}\left\{J_0(|\omega\mz_m-\eta\mx|)^2-\sum_{s=1}^{2}\left(\frac{\omega\mz_m-\eta\mx}{|\omega\mz_m-\eta\mx|}\cdot\me_s\right)^2J_1(|\omega\mz_m-\eta\mx|)^2\right\}\right|.
  \end{equation}
\end{thm}
\begin{proof}
  Since $N$ is sufficiently large and $\vv_j=-\vt_j$ for all $j$, by comparing (\ref{ED}) and (\ref{SVD}), we can observe that
  \[\mathbb{W}(\mx;\eta)=\left|\sum_{m=1}^{3M}\langle\mW(\mx;\eta),\mU_m\rangle\langle\mW(\mx;\eta),\overline{\mV}_m\rangle\right|\approx\left|\sum_{m=1}^{M}\sum_{s=1}^{3}\langle\mW(\mx;\eta),\mE_s(\mz_m;\omega)\rangle^2\right|.\]

  First, applying Lemma \ref{Bessel} yields
  \[\langle\mW(\mx;\eta),\mE_1(\mz_m;\omega)\rangle=\frac{1}{N}\sum_{n=1}^{N}e^{i\vt_n\cdot(\omega\mz_m-\eta\mx)}=\frac{1}{2\pi}\int_{\mathbb{S}^1}e^{i\vt\cdot(\omega\mz_m-\eta\mx)}d\vt=J_0(|\omega\mz_m-\eta\mx|).\]
  Thus
  \begin{equation}\label{term1}
  \langle\mW(\mx;\eta),\mE_1(\mz_m;\omega)\rangle^2=J_0(|\omega\mz_m-\eta\mx|)^2.
  \end{equation}
  Next, applying Lemma \ref{Bessel} again, we can evaluate
  \begin{align*}
    \langle\mW(\mx;\eta),\mE_2(\mz_m;\omega)\rangle&=\frac{1}{N}\sum_{n=1}^{N}(\vt_n\cdot\me_1)e^{i\vt_n\cdot(\omega\mz_m-\eta\mx)}=\frac{1}{2\pi}\int_{\mathbb{S}^1}(\vt\cdot\me_1)e^{i\vt\cdot(\omega\mz_m-\eta\mx)}d\vt\\
    &=i\left(\frac{\omega\mz_m-\eta\mx}{|\omega\mz_m-\eta\mx|}\cdot\me_1\right)J_1(|\omega\mz_m-\eta\mx|).
  \end{align*}
  Hence,
  \begin{equation}\label{term2}
    \langle\mW(\mx;\eta),\mE_2(\mz_m;\omega)\rangle^2=-\left(\frac{\omega\mz_m-\eta\mx}{|\omega\mz_m-\eta\mx|}\cdot\me_1\right)^2 J_1(|\omega\mz_m-\eta\mx|)^2.
  \end{equation}
  Similarly, we can get
  \begin{equation}\label{term3}
    \langle\mW(\mx;\eta),\mE_3(\mz_m;\omega)\rangle^2=-\left(\frac{\omega\mz_m-\eta\mx}{|\omega\mz_m-\eta\mx|}\cdot\me_2\right)^2 J_1(|\omega\mz_m-\eta\mx|)^2.
  \end{equation}

  By combining (\ref{term1}), (\ref{term2}), and (\ref{term3}), we can obtain
  \begin{multline*}
    \sum_{m=1}^{3M}\langle\mW(\mx;\eta),\mU_m\rangle\langle\mW(\mx;\eta),\overline{\mV}_m\rangle\\
    =\sum_{m=1}^{M}\left\{J_0(|\omega\mz_m-\eta\mx|)^2-\sum_{s=1}^{2}\left(\frac{\omega\mz_m-\eta\mx}{|\omega\mz_m-\eta\mx|}\cdot\me_s\right)^2J_1(|\omega\mz_m-\eta\mx|)^2\right\}.
  \end{multline*}
  Hence, (\ref{StructureImagingfunction}) can be derived. This completes the proof.
\end{proof}

\begin{rem}\label{Remark}
  From the relationship (\ref{StructureImagingfunction}), we can observe that since $J_0(x)$ has its maximum value $1$ at $x=0$ and $J_1(0)=0$, $\mathbb{W}(\mx;\eta)$ will plot the magnitude of $1$ at $\mx=(\omega/\eta)\mz_m$ instead of the true locations $\mz_m$. This is the reason the inexact locations of cracks are extracted via subspace migration when inaccurate value of applied frequency applied. Note that if $\eta=\omega$, one can extract true locations of $\Sigma_m$. Furthermore, if $\eta>\omega$, identified locations will be concentrated to the origin. Otherwise, if $\eta<\omega$, identified locations will be located far from the origin. See FIG. \ref{Locations}.
\end{rem}

\begin{figure}[h]
\begin{center}
  \includegraphics[width=0.5\textwidth]{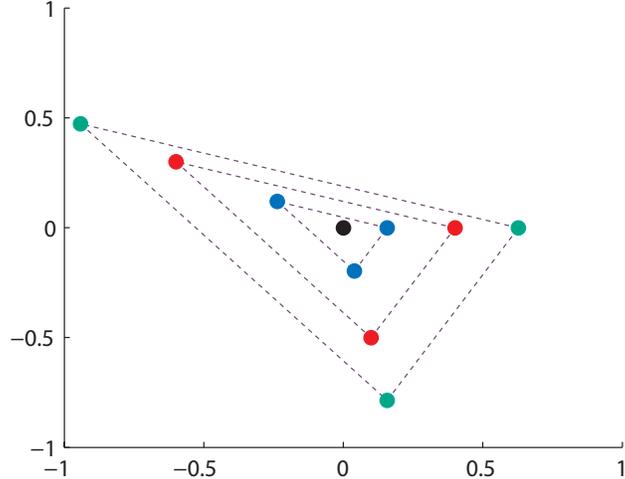}
  \caption{\label{Locations}Description of identified locations in Remark \ref{Remark}. Black-colored circle is origin, red-colored circles are true locations of $\Sigma_m$, blue-colored circles are identified locations of $\Sigma_m$ when $\eta>\omega$, and green-colored circles are identified locations of $\Sigma_m$ when $\eta<\omega$.}
\end{center}
\end{figure}

\section{Simulation results}\label{sec:5}
Some result of numerical simulations are exhibited in this Section in order to examine validation of Theorem \ref{TheoremFrequency}. For performing, we consider the detection of three different small inhomogeneities $\Sigma_m$. The radius $r$ of all $\Sigma_m$ are equally set to $0.05$, and parameters $\eps_0$ and $\mu_0$ are chosen as $1$. Locations $\mz_m$ are selected as $\mz_1=[0.4,0]^T$, $\mz_2=[-0.6,0.3]^T$, and $\mz_3=[0.1,-0.5]^T$. The incident directions $\vt_l$ are selected as
\[\vt_l=\left[\cos\frac{2\pi l}{N},\sin\frac{2\pi l}{N}\right]^T\quad\mbox{for}\quad l=1,2,\cdots,N.\]
In every examples, a white Gaussian noise with $20$ dB Signal-to-Noise Ratio (SNR) is added via the Matlab command \texttt{awgn} included in the signal processing package.

FIG. \ref{Result1} shows the maps of $\mathbb{W}(\mx;\eta)$ when MSR matrix is generated with $N=16$, $\lambda=0.4$, and $\eps_m=\mu_m=5$, $m=1,2,3$. As expected, although unexpected artifacts disturb identification, locations of $\Sigma_m$ are clearly identified for any value of $\eta$. Furthermore, on the basis of Theorem \ref{TheoremFrequency}, since locations of $(\omega/\eta)\mz_m$ are identified via $\mathbb{W}(\mx;\eta)$, extracted locations of $\Sigma_m$ are scattered when $\eta<\omega$ and concentrated when $\eta>\omega$. Notice that since the true value of $\omega$ is $\omega\approx15.7080$, very accurate locations of $\Sigma_m$ are identified via the map of $\mathbb{W}(\mx;15)$.

\begin{figure}[h]
\includegraphics[width=0.495\textwidth]{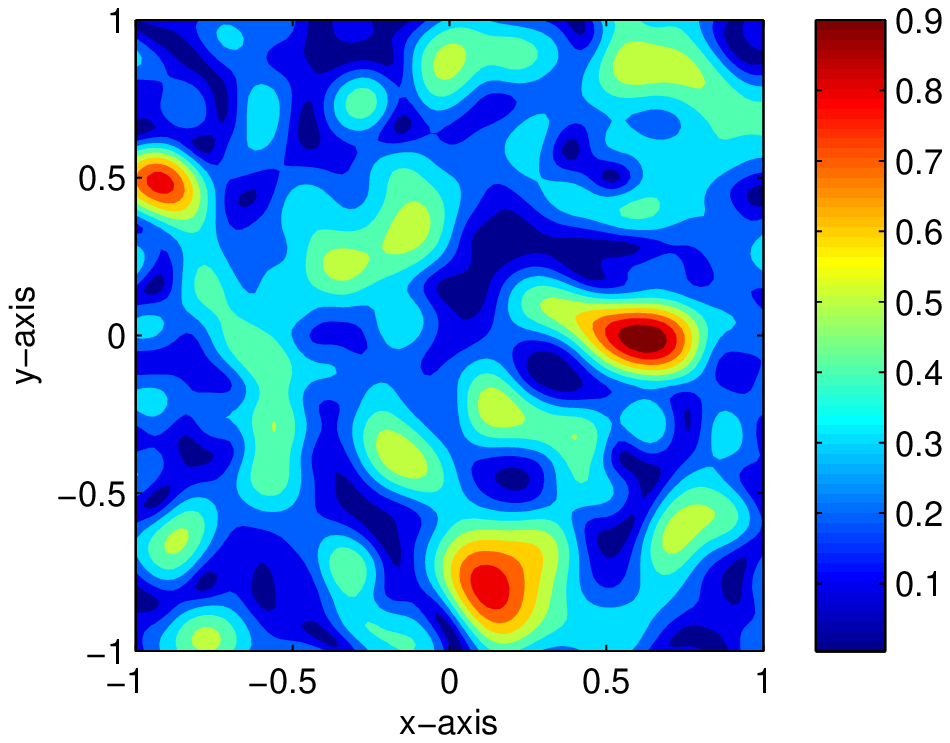}
\includegraphics[width=0.495\textwidth]{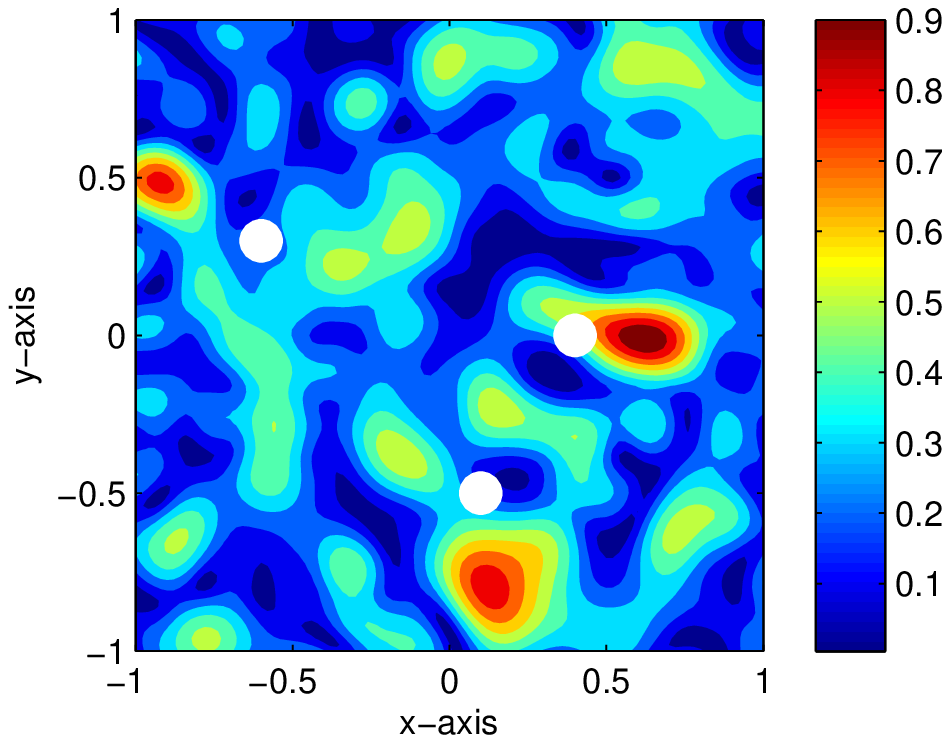}\\
\includegraphics[width=0.495\textwidth]{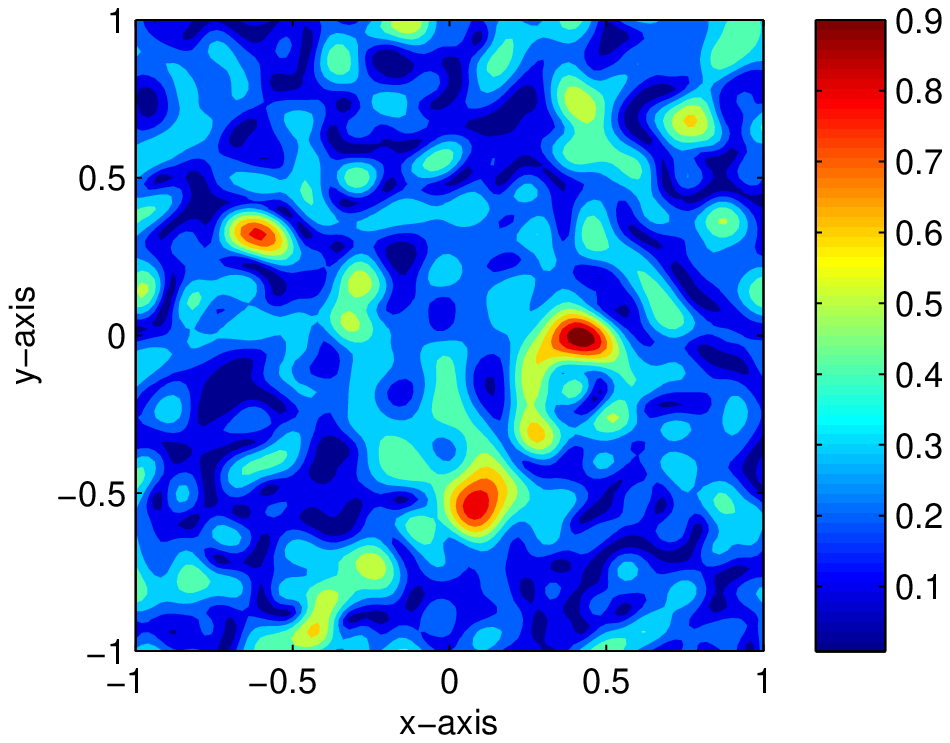}
\includegraphics[width=0.495\textwidth]{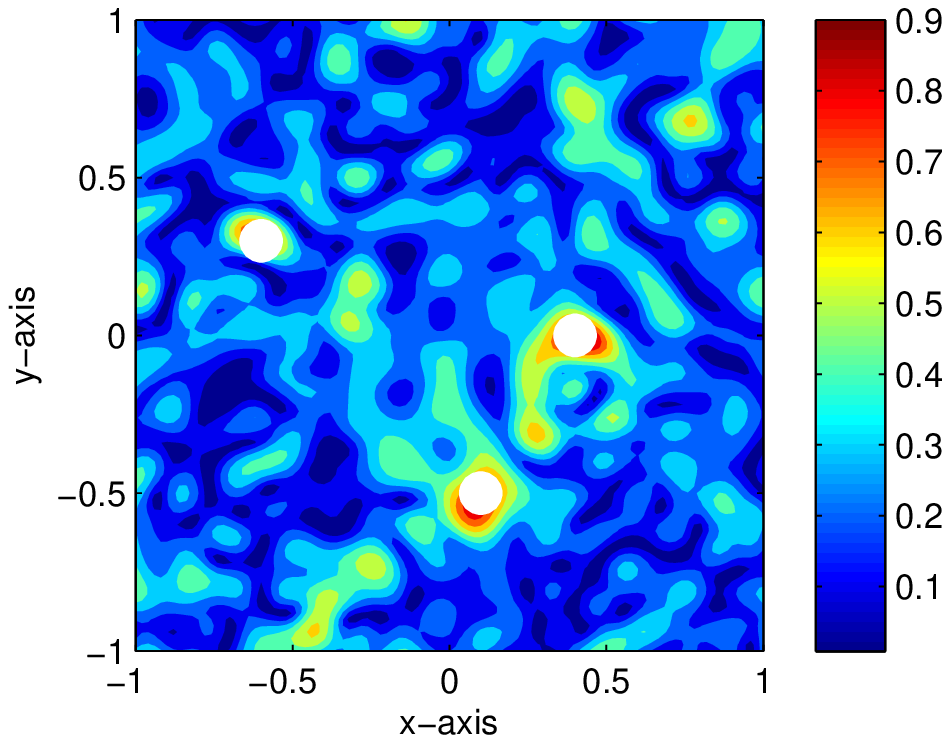}\\
\includegraphics[width=0.495\textwidth]{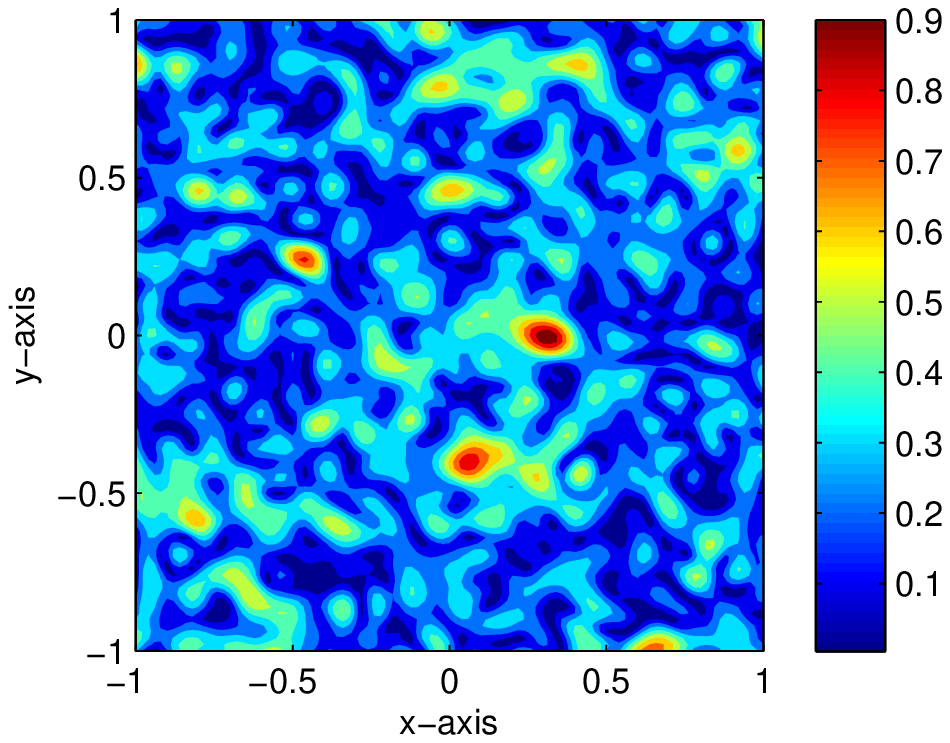}
\includegraphics[width=0.495\textwidth]{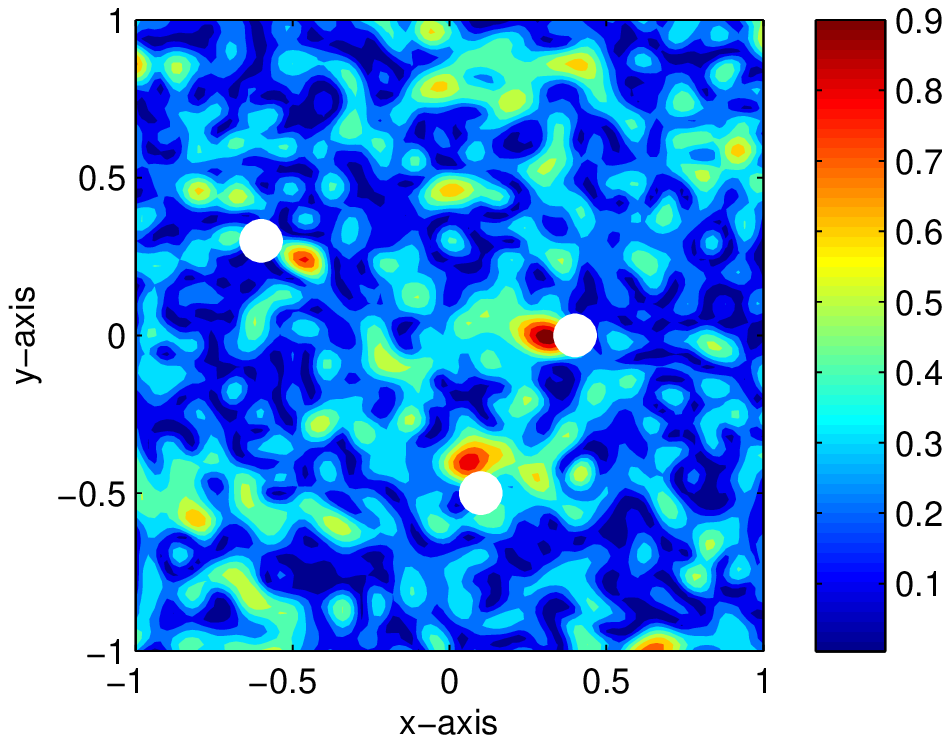}
\caption{\label{Result1}Maps of $\mathbb{W}(\mx;\eta)$ for $\eta=10$ (top, left), $\eta=15$ (middle, left), and $\eta=20$ (bottom, right) when $\omega=2\pi/0.4$. White-colored circles in the maps of $\mathbb{W}(\mx;\eta)$ are true locations of $\Sigma_m$ (right column).}
\end{figure}

FIG. \ref{Result2} shows the maps of $\mathbb{W}(\mx;\eta)$ when MSR matrix is generated with $N=16$, $\lambda=0.2$, and different material properties $\eps_1=\mu_1=5$, $\eps_2=\mu_2=2$, and $\eps_3=\mu_3=7$. Similar to the results in FIG. \ref{Result1}, we can recognize the existence of $\Sigma_m$ but huge amount of artifacts impede identification. Note that since the true value of $\omega$ is $\omega\approx31.4160$, locations of $\Sigma_m$ can be identified accurately via the map of $\mathbb{W}(\mx;30)$. However, we have no \textit{a priori} information of true value of $\omega$, locations of $\Sigma_m$ cannot be identified exactly at this stage.

\begin{figure}[h]
\includegraphics[width=0.495\textwidth]{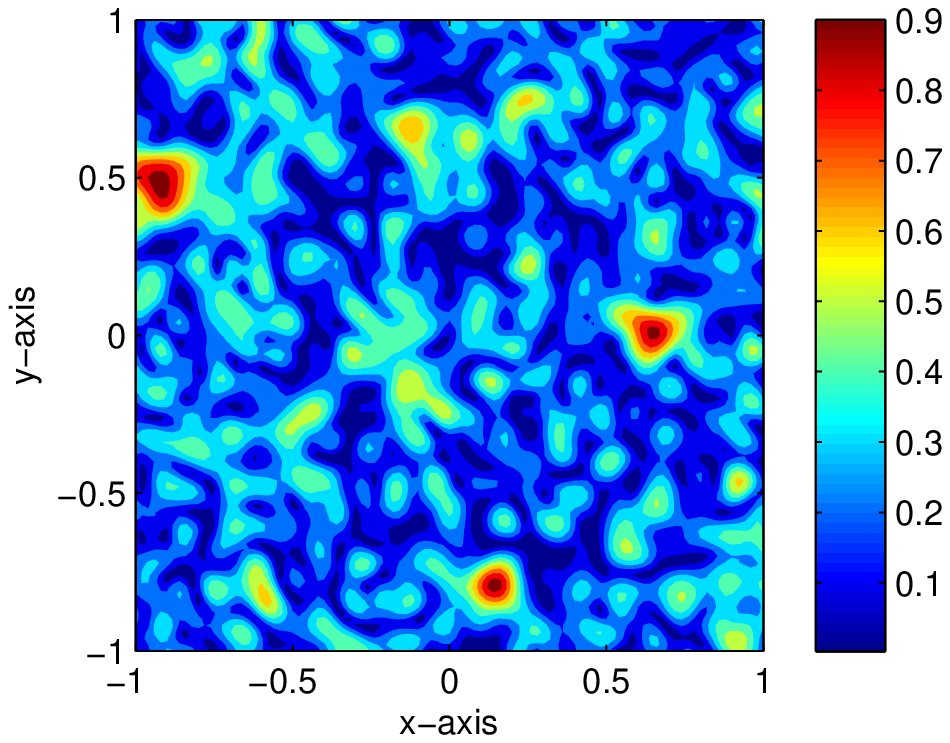}
\includegraphics[width=0.495\textwidth]{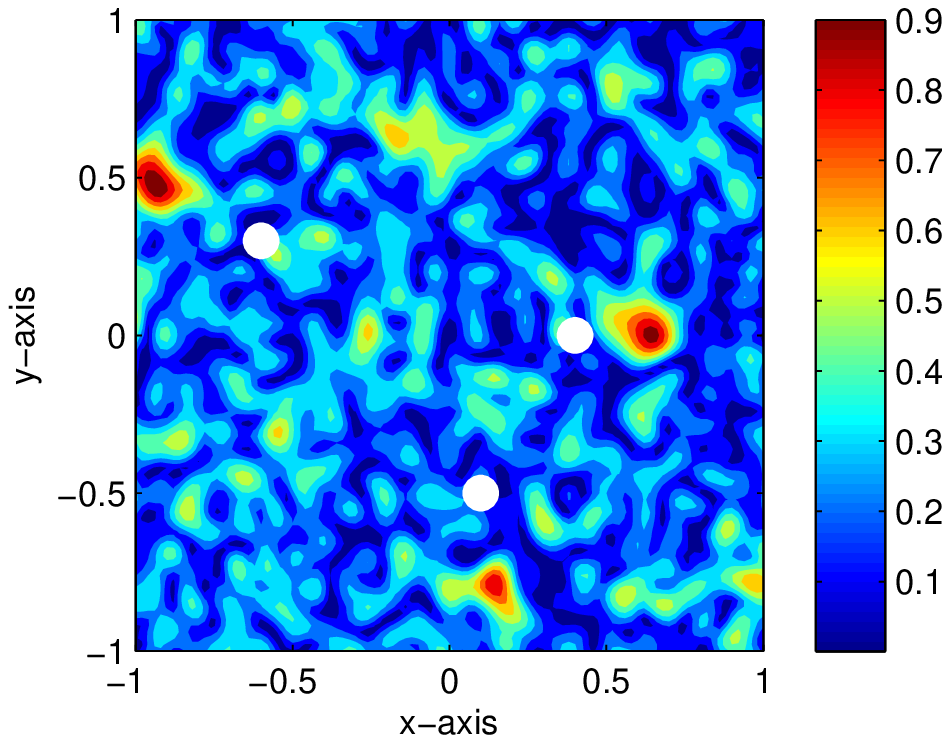}\\
\includegraphics[width=0.495\textwidth]{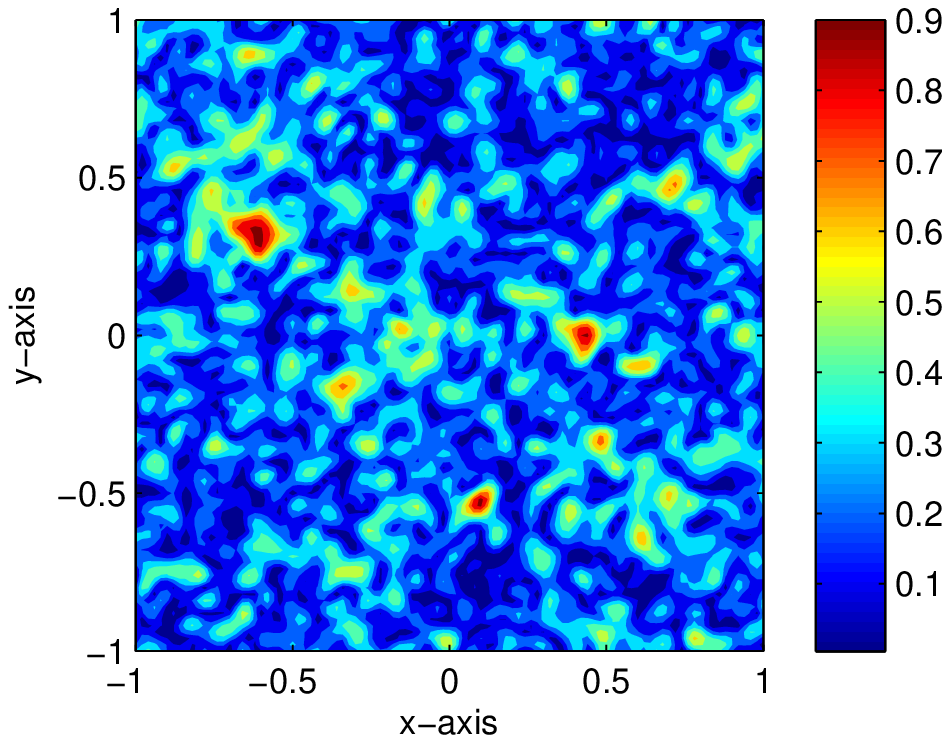}
\includegraphics[width=0.495\textwidth]{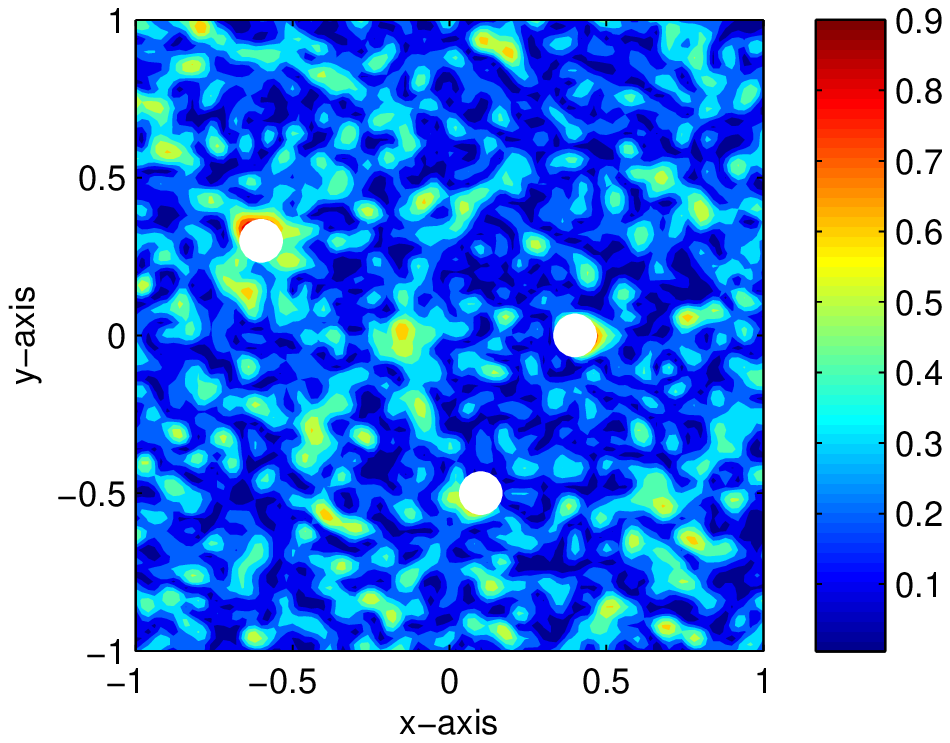}\\
\includegraphics[width=0.495\textwidth]{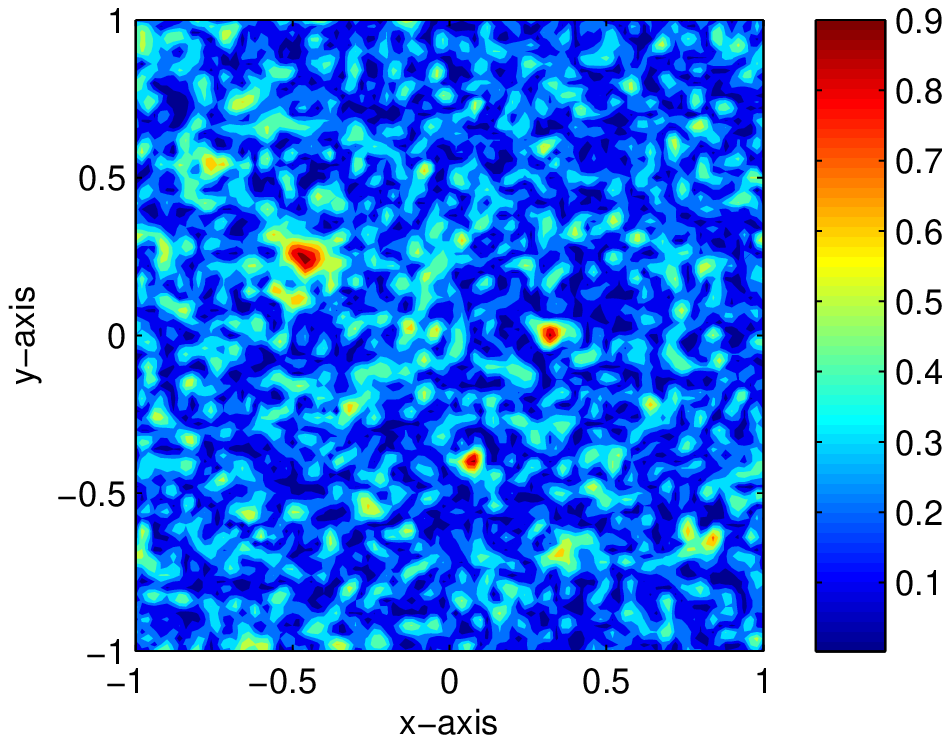}
\includegraphics[width=0.495\textwidth]{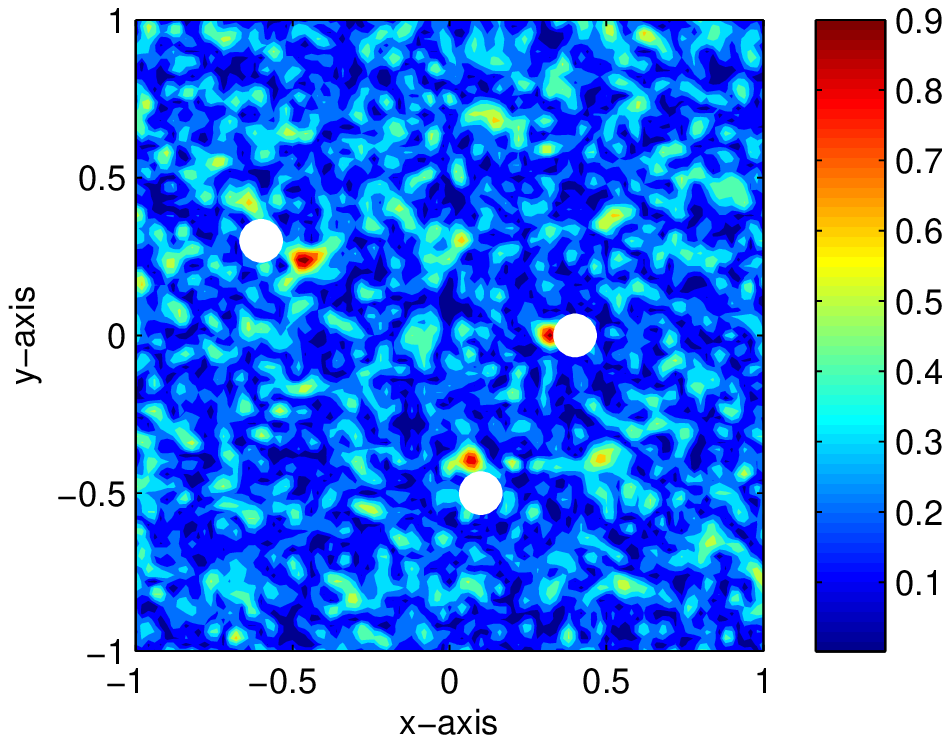}
\caption{\label{Result2}Maps of $\mathbb{W}(\mx;\eta)$ for $\eta=20$ (top, left), $\eta=30$ (middle, left), and $\eta=40$ (bottom, left) when $\omega=2\pi/0.2$. White-colored circles in the maps of $\mathbb{W}(\mx;\eta)$ are true locations of $\Sigma_m$ (right column).}
\end{figure}

\section{Conclusion}\label{sec:6}
Based on the asymptotic expansion formula of far-field pattern in the existence of small electromagnetic inhomogeneities, the structure of subspace migration imaging function is investigated when the applied frequency is inexact to the true one. Based on the relationship with Bessel functions of order zero and one of the first kind, we have confirmed the reason as to why the locations of inhomogeneities identified inaccurately.

In this paper, we considered the detection of inhomogeneities in the full-view inverse scattering problem. Based on the difficulties from \cite{PL2,P-SUB5,AIL1}, extension to the limited-view or half-space problem will be an remarkable research subject. Furthermore, we expect it can be extended to the various inverse scattering problem in three-dimension.

\begin{acknowledgments}
This research was supported by the Basic Science Research Program through the National Research Foundation of Korea (NRF), funded by the Ministry of Education(No. NRF-2014R1A1A2055225), and the research program of Kookmin University in Korea.
\end{acknowledgments}

\end{document}